\documentclass[12pt,oneside]{article}
\usepackage{amsmath}
\usepackage{amsfonts}
\usepackage{amssymb}
\usepackage{graphicx}

\usepackage{latexsym}
\usepackage{amssymb, latexsym}
\usepackage{amsmath}
\usepackage{mathrsfs}
\usepackage{amsxtra}
\usepackage{amsthm}
\usepackage{setspace}
\usepackage{epsfig}
\usepackage{geometry}
\usepackage{graphicx}
\usepackage{blindtext}

\newtheorem{theorem}{Theorem}[section]
\newtheorem{corollary}[theorem]{Corollary}

\newtheorem{lemma}[theorem]{Lemma}

\newtheorem{proposition}[theorem]{Proposition}

\theoremstyle{definition}
\newtheorem{definition}[theorem]{Definition}

\newtheorem{remark}[theorem]{Remark}

\date{}

\begin{document}

\author{Sergei Artamoshin
\thanks{The author wants to thank professor Adam Koranyi, who pointed out that the function $\omega(x,y)$ and its radialization over sphere are eigenfunctions of the hyperbolic Laplacian. The author is also thankful to professor J\a'ozef Dodziuk for his interest to this paper and for his comments.}\\ \\Department of Mathematical Science, Central Connecticut State University\\New Britain, USA\\ \\sartamoshin@gmail.com%
}

\title{\textbf{One Radius Theorem For A Radial Eigenfunction Of A Hyperbolic Laplacian}}

\maketitle

\begin{abstract}
Let us fix two different radial eigenfunctions of a hyperbolic Laplacian and assume that both of them have the same value at the origin. Both eigenvalues can be complex numbers. The main goal of this paper is to estimate the lower bound for the interval (0,T], where these two eigenfunctions must assume different values at every point. We shall see that T is a function of two different eigenvalues corresponding to the given pair of radial eigenfunctions. On the other hand, we shall see that at every fixed point and for the value already assumed by a radial eigenfunction at the fixed point, there are infinitely many other radial eigenfunctions, assuming the same value at the fixed point and satisfying the same initial condition.

\end{abstract}

\bigskip

\section{Introduction}
Consider the set of all radial eigenfunctions of the Hyperbolic Laplacian assuming the value 1 at the origin, i.e. the set of all solutions for the following system
\begin{equation}\label{Intro-2}
\left\{
  \begin{array}{ll}
     & \hbox{$\varphi^{''}(r)+\frac{k}{\rho}\coth\left(\frac{r}{\rho}\right)\varphi^{'}(r)+\lambda\varphi(r)=0$, $\lambda\in\mathbf{C}$;} \\
     & \hbox{$\varphi(0)=1$\,,}
  \end{array}
\right.
\end{equation}
written in the geodesic polar coordinates of the hyperbolic space of constant sectional curvature $\kappa=-1/\rho^2$.
Recall that for every $\lambda\in\mathbf{C}$ there exists a unique solution $\varphi_\lambda(r)$ such that $\varphi_\lambda(0)=1$, see \cite{Chavel} (p. 272). If we choose two radial eigenfunctions $\varphi_\mu(r)$ and $\varphi_\nu(r)$ such that $\varphi_\mu(0)~=~\varphi_\nu(0)~=~1$, then, how many values of $r>0$ such that $\varphi_\mu(r)=\varphi_\nu(r)$ are sufficient to ensure that $\mu=\nu$ and $\varphi_\mu(r)\equiv\varphi_\nu(r)$? We shall see that we need only one such a value of $r$ if the value is small enough. In other words, if $\mu\neq\nu$, then there exists an interval $(0,T(\mu,\nu)]$, such that $\varphi_\mu(r)\neq\varphi_\nu(r)$ for all $r\in(0,T(\mu,\nu)]$. In particular, we prove that if $\mu\neq \nu$ are real and $\mu, \nu\leq k^2/4$ then $\varphi_\mu(r)\neq\varphi_\nu(r)$ for all $r\in(0,\infty)$. So, if $\mu, \nu\leq k^2/4$ and $\varphi_\mu(r)=\varphi_\nu(r)$ just for one arbitrary $r>0$, then $\mu=\nu$ and $\varphi_\mu(r)\equiv\varphi_\nu(r)$.\\
On the other hand, according to Corollaries~\ref{Many_Eigefcns_Cor} and~\ref{Many_Eigefcns_Cor2} from p.~\pageref{Many_Eigefcns_Cor}, for every fixed point $r_0\in[0,\infty)$, and for a fixed radial eigenfunction $\varphi_\nu(r)$, assuming the value one at the origin, there are infinitely many other radial eigenfunctions assuming the value one at the origin and the value $\varphi_\nu(r_0)$ at the point $r_0$.

\section{Some geometric preliminaries.}

\subsection {Notations and basic definitions}

For the future reference, let us introduce the following notation. Let $H_\rho^n$ and $B^n_\rho$ be the half-space model and the ball model, respectively, of a hyperbolic $n-$dimensional space with a constant sectional curvature $\kappa=1/{\rho^2}<0$. Recall that the half-space model $H_\rho^n$ consists of the open half-space of points $(x_1, ..., x_{n-1}, t)$ in $R^n$ for all $t>0$ and the metric is given by $(\rho/t)|ds|$, where $|ds|$ is the Euclidean distance element. The ball model $B^n_\rho$ consists of the open unit ball $|X|^2+T^2<\rho^2, (X,T)=(X_1, ..., X_{n-1}, T)$ in $R^n$, and the metric for this model is given by $2\rho^2|ds|/(\rho^2-|X|^2-T^2)$.

Recall also that in a hyperbolic space a Laplacian can be represent as
\begin{equation}\label{Uniq-6}
\begin{split}
    \triangle
    & =\frac{1}{\rho^2}\left[ t^2\left( \frac{\partial^2}{\partial x_1^2}+...+\frac{\partial^2}{\partial x_{n-1}^2} + \frac{\partial^2}{\partial t^2} \right)-(n-2)t\frac{\partial}{\partial t} \right]
    \\& =\frac{\partial^2}{\partial r^2}+
    \frac{n-1}{\rho}\coth\left(\frac{r}{\rho}\right)\frac{\partial}{\partial r}+\triangle_{S(0;r)}\,,
\end{split}
\end{equation}
where the first expression is the hyperbolic Laplacian in $H^n_\rho$ expressed by using Euclidean rectangular coordinates and the second expression represents  the hyperbolic Laplacian in $B^n_\rho$ expressed in the geodesic hyperbolic polar coordinates, where $\triangle_{S(0;r)}$ is the Laplacian on the geodesic sphere of a hyperbolic radius $r$ about the origin.

The next step is to introduce the function $\omega$ and show that $\omega^\alpha$ as an eigenfunction of the Hyperbolic Laplacian in $B^n_{\rho}$.

\begin{proposition}\label{eigenfunction-omega-proposition}
Let $u$ be any point of $S(\rho)$ and $m=(X,T)\in B(\rho)$, where $S(O;\rho)$ and $B(\rho)$ are the Euclidean sphere and ball, respectively, both of the same radius $\rho$ centered at the origin $O$. Let $\eta=|m|$ be the Euclidean distance between the origin $O$ and a point $m$, $k=n-1$ and
\begin{equation}\label{Uniq-11}
    \omega=\omega(u,m)=\frac{\rho^2-\eta^2}{|u-m|^2}=\frac{|u|^2-|m|^2}{|u-m|^2}\,.
\end{equation}
Then
\begin{equation}\label{Omega-as-Eigenfunction}
    \triangle_m \omega^\alpha+\frac{\alpha k-\alpha^2}{\rho^2}\, \omega^\alpha=0\,.
\end{equation}
\end{proposition}

\begin{proof}[Proof of the Proposition.]
Notice that
\begin{equation}\label{Uniq-12}
    \triangle t^\alpha+\frac{\alpha k-\alpha^2}{\rho^2} t^\alpha=0
\end{equation}
and the relationship between $H^n_\rho$ and $B^n_\rho$ is given by Cayley Transform \\
 $K:B^n_\rho\rightarrow H^n_\rho$ expressed by the following formulae
\begin{equation}\label{Uniq-13}
\begin{split}
& x=\frac{2\rho X}{|X|^2+(T-\rho)^2}
\\& t=\frac{\rho^2-|X|^2-T^2}{|X|^2+(T-\rho)^2}=\frac{\rho^2-\eta^2}{|u-m|^2}=\omega(u,m)\,,
\end{split}
\end{equation}
where $u=(0,\rho)\in\partial B(O;\rho)=S(O;\rho)$. Since Cayley transform is an isometry, an orthogonal system of geodesics defining the Laplacian in $B^n_\rho$ is mapped isometrically to an orthogonal system of geodesics in $H^n_\rho$ and then,
\begin{equation}\label{Uniq-14}
    \triangle t^\alpha=\triangle_m\omega^\alpha(u,m)\quad \text{and}\quad t^\alpha=\omega^\alpha(u,m)\,.
\end{equation}
Therefore, equation \eqref{Omega-as-Eigenfunction} is precisely equation \eqref{Uniq-12} written in $B^n_\rho$, which completes the proof of Proposition~\ref{eigenfunction-omega-proposition}.
\end{proof}



\subsection {Explicit representation of radial eigenfunctions.}

In this subsection we shall see that every radial eigenfunction in $B^{k+1}_{\rho}$ depending only on the distance from the origin has an explicit integral representation.

\begin{definition}\label{sharp-definition} Recall that $B^{k+1}_\rho$ is the ball model of the hyperbolic space~$H^{k+1}_\rho$ with the sectional curvature $\kappa=-1/\rho^2$ and let $B^{k+1}(O,\rho)$ be the Euclidean ball of radius $\rho$ centered at the origin and represents the ball model $B^{k+1}_\rho$. Suppose that $f$ is a function on $B^{k+1}_\rho$. We define its radialization about the origin $O$, written $f^\sharp_O (m)$, by setting
\begin{equation}\label{Radialization-Definition}
    f^\sharp_O(m)=\frac{1}{|S^k(|m|)|}\int\limits_{S^k(|m|)} f(m_1) dS_{m_1}\,,
\end{equation}
where the integration is considered with respect to the measure on $S^k(|m|)$ induced by the Euclidean metric of $\mathbb{R}^n\supset B^{k+1}(O,\rho)$; recall also that $|m|$ is the Euclidean distance between the origin $O$ and a point $m\in B^{k+1}_\rho$; $|S^k(|m|)|$ is the Euclidean volume of $S^k(|m|)$.
\end{definition}

\begin{lemma}\label{Laplacian-Commutes-lemma}
\begin{equation}\label{Laplacian-Commutes}
    \triangle_m f^\sharp_O(m)=\frac{1}{|S^k(|m|)|}\int\limits_{S^k(|m|)}\triangle_{m_1} f(m_1) dS_{m_1}\,.
\end{equation}
\end{lemma}

\begin{proof}[Proof of Lemma~\ref{Laplacian-Commutes-lemma}]
Let us introduce two measures on $S^k$. One is the Lebesque measure of the sphere, i.e. the Riemannian volume of the flat metric of the ambient space restricted to $S^k$ and the other one is the Haar measure of $\text{SO}(k+1,R)$ pushed forward to the sphere by the map $T\rightarrow Tv_0$, where $v_0$ is a fixed vector on the sphere and $T\in\text{SO}(k+1,R)$. These two measures agree up to constant multiple factor so that these two averaging procedures must be the same. It is clear that Laplacian commutes with averaging over $\text{SO}(k+1,R)$, since any individual isometry of $\mathbb{R}^{k+1}$ commutes with Laplacian. Therefore, Laplacian must commute with radialization defined in \eqref{Radialization-Definition}. This completes the proof of Lemma~\ref{Laplacian-Commutes-lemma}.
\end{proof}



The next step is to obtain the explicit representation for radial eigenfunctions.

\begin{definition} Let $u\in S^k(\rho)$ be a fixed point and $|m_1|=|m|=\eta<\rho=|u|$. Then let us define
\begin{equation}\label{Radialization-of-omega}
    V_\alpha(\eta)=(\omega_O^\alpha)^\sharp(m)=
    \frac{1}{|S(\eta)|}\int\limits_{S(\eta)}\omega^\alpha(u,m_1)dS_{m_1}\,,
\end{equation}
where $\alpha$ is a complex number. Thus, $V_\alpha(\eta)$ is the radialization of $\omega^\alpha(u,m)$ about the origin.
\end{definition}




\begin{theorem}\label{Radial-func-repre-Thoerem} Let $r$ be the hyperbolic distance between the origin $O$ and $S(\eta)$. Then, the following function
\begin{equation}\label{Radial-func-presentation}
    \varphi_\mu(r)= V_\alpha(\eta(r))= V_\alpha\left(\rho\tanh\left(\frac{r}{2\rho}\right)\right)\,,
\end{equation}
where $\mu=(\alpha k-\alpha^2)/\rho^2$, is the unique radial eigenfunction assuming the value $1$ at the origin and corresponding to an eigenvalue $\mu$, i.e.,

\begin{equation}\label{Equation-for-Radi-Eigenfunction}
    \triangle\varphi_\mu(r)+\mu\varphi_\mu(r)=0.
\end{equation}
\end{theorem}

\begin{proof}
    Recall that $\eta=|m|$ is the Euclidean distance between $m=(X,T)\in B(\rho)$ and the origin, while $r=r(m)=r(\eta)$ is the hyperbolic distance between the origin and $m$. Therefore, the relationship between $r$ and $\eta$ is
\begin{equation}\label{Eucli-Hype-Coordinates-relationship}
    r=\rho\ln\frac{\rho+\eta}{\rho-\eta}\quad\text{or}\quad\eta=\rho\tanh\left(\frac{r}{2\rho}\right)\,,
\end{equation}
which justifies the last expression in \eqref{Radial-func-presentation}.

Recall also that according to \eqref{Omega-as-Eigenfunction}, $\omega^\alpha$ is the eigenfunction of the hyperbolic Laplacian with the eigenvalue $(\alpha k-\alpha^2)/\rho^2$. Therefore, according to Lemma \ref{Laplacian-Commutes-lemma}, p.~\pageref{Laplacian-Commutes-lemma}, the radialization of $\omega^\alpha$ defined in \eqref{Radialization-of-omega} is also an eigenfunction with the same eigenvalue.

The last step is to show that the eigenfunction $\varphi_\mu(r)$ assuming the value $1$ at the origin is unique. We obtain the uniqueness by the procedure described in \cite{Olver}, pp.~148-153 or in \cite{Chavel}, p.~272 resulting that the general radial eigenfunction is on the form $c_1f_1(r)+c_2f_2(r)$, where $f_1(r)$ is an entire function of $r$ with $f_1(0)=1$, and $f_2(r)$ is defined for all $r>0$, with a singularity at $r=0$. Therefore, for a radial eigenfunction to have the value $1$ at the origin, we must set $c_2=0$ and then, we can see that the radial eigenfunction is defined uniquely by its value at the origin and by the eigenvalue. Observing that $\varphi_\mu(0)=V_{\alpha}(0)=1$ completes the proof of Theorem~\ref{Radial-func-repre-Thoerem}.
\end{proof}

\subsection{One Radius theorem for $\omega$.} 

Recall that One Radius Theorem for the function $\omega$ introduced in \eqref{Uniq-11} and considered as a function of two arbitrary non-equal points in Euclidean space was obtained in~\cite{Artamoshin_Sphe_Ratio}. In this subsection we are going to reproduce the statement of this theorem and use it later as the basic result in the proof of One Radius Theorem for radial eigenfunctions of the hyperbolic Laplacian.

\begin{theorem}[One Radius Theorem for $\omega$]\label{Basic-theorem}

Let $S^k$,
$k\in\mathbb{N}$, be a $k$-dimensional sphere of radius $R$
centered at the origin $O$, $x, y\in\mathbb{R}^{k+1}$ such that $r=|x|\neq|y|=R$ and $$\omega(x,y)=\left|\frac{|x|^2-|y|^2}{|x-y|^2}\right|\,.$$
Until further notice we assume that $R$ and a point $x\notin S^k$ are fixed.
Then, the following Statements hold.

\begin{description}

  \item[(A)]\label{Sphe-pro-direct-Statement} If $\alpha, \beta\in\mathbb{C}$ and $\alpha+\beta=k$, then
\begin{equation}\label{Sphe-Pro-Statement-2}
    \int\limits_{S^k}\omega^\alpha dS_y=
    \int\limits_{S^k}\omega^\beta dS_y\,.
\end{equation}

  \item[(B)]\label{sphe-pro-inverse-real-Statement} If $\alpha, \beta$ are real, then
\begin{equation}\label{Sphe-Pro-Statement-3}
     \int\limits_{S^k}\omega^\alpha dS_y=
    \int\limits_{S^k}\omega^\beta dS_y\quad \text{implies}\quad \alpha+\beta=k\,\,\,
    \text{or}\,\,\,\alpha=\beta  \,.
\end{equation}

  \item[(C)]\label{sphe-pro-liuville-Statement} For every $\beta\in \mathbb{C}$ there are infinitely many numbers $\alpha\in\mathbb{C}$ such that
\begin{equation}\label{Sphe-Pro-Statement-4}
    \int\limits_{S^k}\omega^\alpha dS_y=
    \int\limits_{S^k}\omega^\beta dS_y\,.
\end{equation}

    \item[(D)]\label{sphe-pro-inverse-complex-Statement} Suppose that for the fixed point $x\notin S^k$
\begin{equation}\label{Sphe-Pro-Statement-5-1}
    \max\{|\Im(\alpha)|, |\Im(\beta)|\}\leq \left.\frac\pi2 \right/ \ln\frac{R+r}{|R-r|}\,.
\end{equation}
Then
\begin{equation}\label{Sphe-Pro-Statement-6}
\int\limits_{S^k} \omega^\alpha dS_y =
\int\limits_{S^k} \omega^\beta dS_y \quad\text{implies}\quad
\alpha+\beta=k\quad\text{or}\quad\alpha=\beta \,.
\end{equation}

\end{description}
\end{theorem}

\begin{remark}
    Statement~(C) shows that the implication introduced in \eqref{Sphe-Pro-Statement-3} fails if we just let $\alpha$ and $\beta$ be complex. However, a certain additional restriction formulated in Statement~(D) for the complex numbers $\alpha, \beta$ allows us to obtain precisely the same implication as in \eqref{Sphe-Pro-Statement-3}.
\end{remark}

\begin{remark}
    Observe that Statement~(B) is a special case of Statement~(D). Indeed, if $\alpha, \beta$ are real, then \eqref{Sphe-Pro-Statement-5-1} holds for every $r\neq R$.
\end{remark}

\section {One Radius Theorem for radial eigenfunctions.}

In this section we shall see that if $\varphi_\mu(r)=\varphi_\nu(r)$ for $r=0$ and for some $r=r_0$ chosen sufficiently close to the origin, then $\mu=\nu$ and $\varphi_\mu(r)=\varphi_\nu(r)$ for all $r\in[0,\infty)$. In other words, we shall introduce a condition for an eigenvalue under which the eigenfunction is uniquely determined just by its value at $\emph{one}$ point sufficiently close to the origin. Moreover, we shall see that with an additional condition for the eigenvalue, a radial eigenfunction is uniquely determined by its value at an arbitrary point.

\begin{proposition}\label{Parabola-Strip-equivalence-Lemma} Let $\alpha\in\mathbb{C}$ and $\mu=\Phi(\alpha)$, where $\Phi:\mathbb{C}\rightarrow \mathbb{C}$ is given by $\Phi(\alpha)=-\kappa(\alpha k-\alpha^2)$, where $\kappa<0, k>0$ are constants. Then
\begin{equation}\label{One_Rad_1}
   \Im(\alpha)\in[-p,p]\iff \mu\in\{\text{Shaded area inside the parabola bellow}\}
\end{equation}

\begin{figure}[!h]
    \centering
    \epsfig{figure=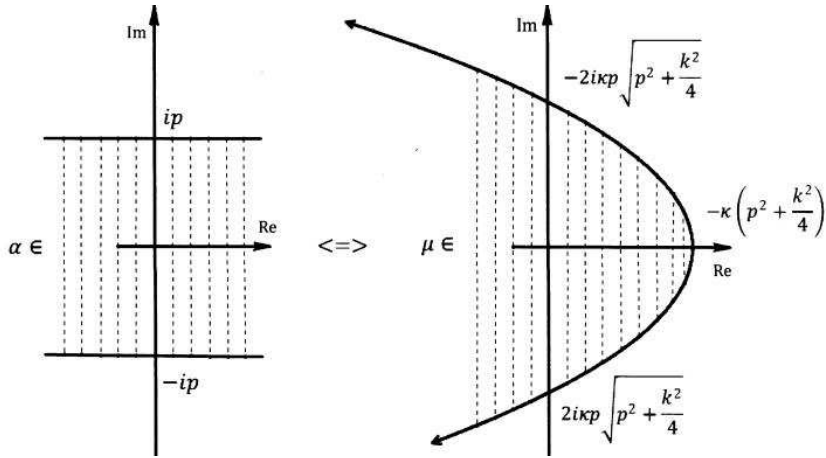,height=6cm}\\
  \caption{Parabola and Strip equivalence.}\label{Parabola-Strip-Equiv}
\end{figure}

\end{proposition}


\begin{proof}[Proof of Proposition \ref{Parabola-Strip-equivalence-Lemma}.]  Let $\alpha=a+ib; a_1=a-k/2$ and $t=a_1+ib$. Then $\alpha=a+ib=k/2+t$. The direct computation shows that $\Phi(k/2+t)=\Phi(k/2-t)$, which means that the image of the following strip
\begin{equation}\label{One_Rad_2}
   \Xi(-p,p)=\{\alpha\mid\Im(\alpha)\in[-p,p]\}
\end{equation}
must be the same as the image of the upper half of the strip above. I.e.,
\begin{equation}\label{One_Rad_3}
   \Phi(\Xi(-p,p))=\Phi(\Xi(0,p))\,,
\end{equation}
where
\begin{equation}\label{One_Rad_4}
   \Xi(0,p)=\{\alpha\mid\Im(\alpha)\in[0,p]\}\,.
\end{equation}
To figure out what is the image of $\Xi(0,p)$, it is enough to look at the image of a horizontal line from $\Xi(0,p)$. Let us choose the boundary line: $\alpha=\alpha(a_1)=k/2+ip+a_1$, where $p$ is fixed and $a_1$ serves as a parameter. Then,
\begin{equation}\label{One_Rad_5}
\begin{split}
   & \Phi(\alpha(a_1))=-\kappa\left(\frac{k^2}{4}+p^2 \right)+\kappa a_1^2+2i\kappa p a_1
   \\& =-\kappa\left(\frac{k^2}{4}+p^2\right)+\frac{1}{4\kappa p^2}(\Im\Phi(\alpha))^2+i\Im\Phi(\alpha)\,,
\end{split}
\end{equation}
where, clearly, the imaginary part depends on $a_1$ linearly, while the real part depends quadratically. Therefore, the horizontal line chosen is mapped to a parabola. Any parabola is uniquely determined by three arbitrary points on it. Thus, to determine this parabola, we find all points of intersections with the coordinate axis. If $a_1=0$, then
\begin{equation}\label{One_Rad_6}
   \Phi(\alpha(0))=-\kappa\left(\frac{k^2}{4}+p^2 \right)
\end{equation}
gives the point of intersection the parabola with the horizontal coordinate axis. If $a_1=\pm\sqrt{p^2+k^2/4}$, then
\begin{equation}\label{One_Rad_7}
   \Phi(\alpha(\pm\sqrt{p^2+k^2/4}))=\pm2i\kappa p\sqrt{p^2+k^2/4}\,,
\end{equation}
what is denoted on Figure \ref{Parabola-Strip-Equiv} above. Therefore, both of these two parallel lines $\alpha(a_1)=\pm ip+k/2+a_1$ are mapped to the parabola described above. Notice also that this parabola is symmetric with respect to the horizontal axis. In addition, the smaller the parameter $p$, the narrower the parabola and its tip is closer to the origin.
So, the proof of Proposition~\ref{Parabola-Strip-equivalence-Lemma} is complete.
\end{proof}

\begin{remark}
Note that if the parameter $p$ is zero, then the parabola is folded up to the half of the real line $(-\infty, -\kappa k^2/4]$, which is the image of the real line.
\end{remark}

\begin{theorem}[\textbf{One Radius theorem for radial eigenfunctions}]\label{One-radius-Theorem}

Let $\mu\neq\nu$ and $\varphi_\nu(r)$, $\varphi_\mu(r)$ be two radial eigenfunctions for the hyperbolic Laplacian given by
\begin{equation}\label{One_Rad_8}
    \vartriangle \varphi(r)=\varphi^{''}(r)+\frac{k}{\rho}\coth\left(\frac{r}{\rho}\right)\varphi^{'}(r) \,,\,\text{where}\,\,0\leq r<\infty
\end{equation}
in a ball model of a hyperbolic $(k+1)-$ dimensional space of a constant sectional curvature $\kappa=-1/\rho^2<0$. Let $\varphi_\nu(0)=\varphi_\mu(0)=1$.
Then $\varphi_\nu(r)\neq\varphi_\mu(r)$ for every $r\in(0, \pi\rho/(2p)]$, where $p=\max\{|\Im(\alpha)|, |\Im(\beta)|\}$ and $\alpha, \beta$ are complex numbers related to $\mu, \nu$ by the following quadratic equations
\begin{equation}\label{One_Rad_9}
    \mu=-\kappa(\alpha k-\alpha^2) \quad \text{and} \quad \nu=-\kappa(\beta k-\beta^2)\,.
\end{equation}
\end{theorem}

\begin{remark} Recall that the first quadratic equation above together with \eqref{One_Rad_5} implies that
\begin{equation}\label{One_Rad_10}
    \Im(\alpha)\in [-p, p] \iff
    \Re\mu\leq-\kappa\left(p^2+\frac{k^2}{4}\right)+\frac{1}{4\kappa p^2}(\Im\mu)^2\,,
\end{equation}
i.e., $\mu$ belongs to the inner part of the parabola depending on parameter $p$ and pictured on Figure~\ref{Parabola-Strip-Equiv}, p.~\pageref{Parabola-Strip-Equiv}.

 If $p=0$, we arrive at the case when $\alpha, \beta$ are real or, equivalently, $\nu, \mu\leq -\kappa k^2/4$, which, according to the One Radius Theorem, means that the condition $\mu\neq\nu$ implies that $\varphi_\nu(r)\neq\varphi_\mu(r)$ for every $r\in (0, \pi\rho/(2p)]_{p=0}=(0,\infty)$.
\end{remark}

\begin{proof}[Proof of Theorem~\ref{One-radius-Theorem}.]

Let us assume that $\varphi_\mu(r)=\varphi_\nu(r)$ for some $r\in(0, \pi\rho/(2p)]$. Recall that according to Theorem~\ref{Radial-func-repre-Thoerem} from p.~\pageref{Radial-func-repre-Thoerem}, a radial eigenfunction in $B^{k+1}_\rho$ assuming the value~$1$ at the origin is uniquely defined by its eigenvalue and this eigenfunction can be expressed by using the Euclidean coordinates as
\begin{equation}\label{One_Rad_11}
   \varphi_\mu(r)=V_\alpha(\eta(r))=\frac{1}{|S(\eta)|}\int\limits_{S(\eta)}\omega^\alpha(u,m)dS_m\,,
\end{equation}
where $\mu=-\kappa(\alpha k-\alpha^2);\,\, |m|=\eta;\,\, r=\rho\ln[(\rho+\eta)/(\rho-\eta)]$ and $|u|=\rho$. Let us set $x=u\notin S^k(\eta)$ and $y=m\in S^k(\eta)$ in Statement (D) of Theorem~\ref{Basic-theorem}, p.~\pageref{Basic-theorem}.
Using \eqref{One_Rad_11} we can observe that
\begin{equation}\label{One_Rad_12}
   \varphi_\nu(r)=\varphi_\mu(r)\quad\text{for some}\,\,r=\rho\ln\frac{\rho+\eta}{\rho-\eta}\in(0, \pi\rho/2p]
\end{equation}
is equivalent to say that there exists $\eta=\rho\tanh(r/2\rho)$ such that
\begin{equation}\label{One_Rad_13}
\begin{split}
   & \int\limits_{S(\eta)}\omega^\alpha(u,m)dS_m=\int\limits_{S(\eta)}\omega^\beta(u,m)dS_m\quad\text{and}
   \\& p=\max\{|\Im(\alpha)|,|\Im(\beta)|\}\leq\left.\frac\pi2\right/\ln\frac{\rho+\eta}{\rho-\eta}\,.
\end{split}
\end{equation}
It is clear that \eqref{One_Rad_13}, according to Statement~(D) of Main Theorem~\ref{Basic-theorem}, p.~\pageref{Basic-theorem} yields $\alpha+\beta=k$ or $\alpha=\beta$, which implies that $\mu=\nu$. The last equality contradicts to the assumption of the theorem, which completes the proof of Theorem~\ref{One-radius-Theorem}.
\end{proof}




\begin{corollary}\label{Many_Eigefcns_Cor} 
Using \eqref{One_Rad_10}, we observe that
\begin{equation}\label{One_Rad_14}
   \Omega(p)=
   \left\{\varphi_\mu(r)\mid\varphi_\mu(0)=1\,\,\text{and}\,\,\,\Re\mu\leq-\kappa\left(p^2+\frac{k^2}{4}\right)+\frac{(\Im\mu)^2}{4\kappa p^2} \right\}
\end{equation}
is the set of all radial eigenfunctions such that $\varphi_\mu(0)=1$ and with eigenvalues within the shaded area inside the parabola depending on $p$. The parabola was pictured on Figure~\ref{Parabola-Strip-Equiv}, p.~\pageref{Parabola-Strip-Equiv}. Then the value of $\varphi_\mu(r_0)$ at any point $r_0\in(0,\pi\rho/(2p)]$ determines the radial eigenfunction $\varphi_\mu(r)\in\Omega(p)$ uniquely. Meanwhile, according to Statement (C) of Theorem~\ref{Basic-theorem} form p.~\pageref{Basic-theorem}, there are infinitely many eigenfunctions $\varphi_\nu(r)\notin\Omega(p)$ such that $\varphi_\nu(r_0)=\varphi_\mu(r_0)$.
\end{corollary}

\begin{corollary}\label{Many_Eigefcns_Cor2} 
If $p=0$ or, equivalently, if $\Omega(0)$ is defined as
\begin{equation}\label{One_Rad_15}
   \Omega(0)=
   \left\{\varphi_\mu(r)\mid\varphi_\mu(0)=1\quad\text{and}\quad\mu\leq\frac{-\kappa k^2}{4}\right\}\,,
\end{equation}
then the value $\varphi_\mu(r_0)$ defines $\varphi_\mu(r)\in\Omega(0)$ uniquely for any $r_0\in(0,\infty)$. And again, by Statement (C) of Theorem~\ref{Basic-theorem} form p.~\pageref{Basic-theorem}, there are infinitely many radial eigenfunctions $\varphi_\nu(r)\notin\Omega(0)$ such that $\varphi_\nu(r_0)=\varphi_\mu(r_0)$.
\end{corollary}

\bibliographystyle{IEEEtran}

\end{document}